\documentclass[a4paper,11pt]{amsart}
\usepackage{amsmath,amssymb,epsfig,rotating}
\usepackage{amsfonts}
\usepackage{xcolor}

\evensidemargin \oddsidemargin
\usepackage[all]{xy}
\usepackage{pinlabel}

\usepackage[
  pdftitle={On the number of periodic orbits of Morse-Smale flows on graph manifolds},
  pdfauthor={Bijan Sahamie},
  citecolor=blue,
  colorlinks,
  linkcolor=blue,
  urlcolor=red,
  ]{hyperref}

\title{On the number of periodic orbits of Morse-Smale flows on graph manifolds}

\author{Bijan Sahamie}
\address{Mathematisches Institut der LMU M\"unchen, 
Theresienstrasse 39, 80333 M\"unchen Germany}
\email{sahamie@math.lmu.de}
\urladdr{http://www.math.lmu.de/~sahamie}

\theoremstyle{plain} 
\newtheorem{theorem}{Theorem}[section]   
\newtheorem{lem}[theorem]{Lemma}         
\newtheorem{prop}[theorem]{Proposition}
\newtheorem{cor}[theorem]{Corollary}

\theoremstyle{definition}

\newtheorem*{ackn}{Acknowledgments}
\numberwithin{equation}{section}

\newcommand{\Z}{\mathbb{Z}}
\newcommand{\Q}{\mathbb{Q}}
\newcommand{\R}{\mathbb{R}}
\newcommand{\sone}{\mathbb{S}^1}
\newcommand{\stwo}{\mathbb{S}^2}
\newcommand{\sthree}{\mathbb{S}^3}
\newcommand{\lra}{\longrightarrow}

\newcommand{\co}{\colon\thinspace}
\newcommand{\fgn}{\mathbb{F}^g_e}
\newcommand{\pd}{{\rm PD}}
\newcommand{\xs}{X_\Sigma}
\newcommand{\xf}{X_{\rm fiber}}
\newcommand{\xline}{\underline{X}}
\newcommand{\xpar}{X^{\vert\vert}}
\newcommand{\xort}{X^{\perp}}

\newcommand{\ipiece}{Y_{n_i}^{g_i;k_i}}
\newcommand{\frake}{\mathfrak{e}}
\newcommand{\frakc}{\mathfrak{c}}
\newcommand{\frakh}{\mathfrak{h}}
\DeclareMathSizes{11}{11}{8}{7}

\begin{document}
\begin{abstract} For a closed oriented $3$-manifold $Y$ we 
define $n(Y)$ to be the minimal non-negative number such 
that in each homotopy class of non-singular vector fields 
of $Y$ there is a Morse-Smale vector field with less or equal to 
$n(Y)$ periodic orbits. We combine the 
construction process of Morse-Smale flows given 
in \cite{Dufraine} with handle decompositions of compact 
orientable surfaces to provide an upper bound to the 
number $n(Y)$ for oriented Seifert manifolds and oriented
graph manifolds prime to $\stwo\times\sone$.
\end{abstract}
\maketitle

\fontsize{11}{14}\selectfont

\section{Introduction}\label{parone}
Morse-Smale vector fields have been the focus of intense studies in the past.
They were applied in the investigation of problems of structural stability 
(cf.~for instance \cite{Asimov}).
Especially their dynamical behavior is easy to understand, which makes them
particularly interesting. Now suppose that $Y$ is a graph manifold. 
For an introduction to basic notions on Morse-Smale vector fields we point the
reader to \cite[\S 1]{Yano}.
In
\cite{Yano}, Yano determined which homotopy classes of non-singular vector fields
of $Y$  admit a non-singular Morse-Smale (in the following just nMS) 
representative. As a consequence 
of his work and the work of Wilson from \cite{Wilson}, it follows that for a graph
manifold $Y$ there exists a finite number $n(Y)$ such that in every homotopy
class of non-singular vector fields there is a  Morse-Smale vector field
whose number of periodic orbits is less or equal to $n(Y)$ (see~\cite[Remark~5.2]{Yano}).
Furthermore, Yano remarked there that it would be interesting to determine
these numbers or to find a relation between a homotopy class $h$ of non-singular
vector fields and the number $n(Y,h)$ which is defined as the minimal
number of periodic orbits a nMS vector field in the class $h$ admits.
In \cite[Th\'{e}or\'{e}me~1.1]{Dufraine} the existence of $n(Y)$ was reproved 
by using an essentially different approach. 
In this article, we will give an upper bound for the number $n(Y)$ for both
oriented Seifert manifolds and graph manifolds prime to $\stwo\times\sone$.
\begin{theorem}\label{result} For an oriented Seifert manifold $Y$ with 
genus-$g$ base $\Sigma$, $n$ exceptional orbits, and Euler number $e$, 
we have
$ n(Y)
 \leq
 4g+4n+8
 -
 4\delta_{|e|,1}
 +
 2(1+\delta_{|e|,1})\delta_{g,0}\delta_{n,0}.
$
\end{theorem}
Since graph manifolds are defined by gluing together Seifert pieces 
along toral boundary components (cf.~\cite{Yano}), the techniques 
applied in the proof of
Theorem~\ref{result} can also be applied in the graph manifold setting.
\begin{theorem}\label{thm:graphmanifolds} Let $Y$ be a irreducible 
graph manifold and let 
$Y_1,\dots,Y_l$, $l>1$, be a  
JSJ decomposition of $Y$, where $Y_i$, $i=1,\dots,l$, is a Seifert
manifold over a genus-$g_i$ base with $k_i$ boundary components 
and $n_i$ exceptional orbits. Then the number
$n(Y)$ is less or equal to 
$6+2\cdot\sum_{i=1}^l \bigl(2g_i+2n_i+\delta_{g_i,0}\delta_{n_i,0}+k_i\bigr)$.
\end{theorem}
This statement also provides an upper bound for orientable graph manifolds
prime to $\stwo\times\sone$. Namely, if we define 
$\beta(Y)=2\cdot\sum_{i=1}^l \bigl(2g_i+2n_i+\delta_{g_i,0}\delta_{n_i,0}+k_i\bigr)$ 
for an irreducible graph 
manifold $Y$ (cf.~Theorem~\ref{thm:graphmanifolds})
then the following statement is immediate.
\begin{cor}\label{cor:graphmanifolds} Let $Y$ be an orientable graph 
manifold prime to $\stwo\times\sone$
and denote by $Y_1\#\dots\#Y_n$ a prime decomposition of $Y$. Then the inequality
$n(Y)\leq 6+\sum_{i=1}^n\beta(Y_i)$ holds.
\end{cor}
In fact, the techniques applied here allow us to determine upper bounds 
for $n(Y,h)$ for every homotopy class $h$ which admits
a nMS representative. This is implicit in the present work but 
not explicitly pointed out, because it is just of mild relevance to the proof
of the statements.
\begin{ackn} We thank Hansj\"org Geiges for pointing our interest
to this question.
\end{ackn}

\section{A Sketch of the Construction}\label{sec:asodc}
Given two nowhere vanishing vector fields $X_1$ 
and $X_2$ on a closed oriented $3$-manifold $Y$, 
there are two obstructions to join $X_1$ 
and $X_2$ by a homotopy through nowhere vanishing vector fields. 
The first obstruction is a class in $H^2(Y;\Z)$ and it is denoted by
$d^2(X_1,X_2)$ (cf.~\cite[\S 4.2]{Geiges}).  
It measures the {\it homotopical distance} of the
vector fields $X_1$ and $X_2$ over the $2$-skeleton of $Y$. This
means, if $d^2(X_1,X_2)$ vanishes, then it is possible to homotope
$X_1$  such that, after the homotopy, it coincides with $X_2$ along
the $2$-skeleton of $Y$. The second obstruction is a class in 
$H^3(\sthree;\Z)$ and denoted by $d^3(X_1,X_2)$. It is defined only
in case $d^2(X_1,X_2)$ vanishes. Then $d^3$ determines whether
the homotopy that joins $X_1$ and $X_2$ over $Y^{(2)}$ can be
extended over the $3$-cells (cf.~\cite[\S 4.2]{Geiges}).\vspace{0.3cm}\\
In \cite{Dufraine}, Dufraine observes that the obstruction class $d^2$
can be expressed in terms of the set
\[
  C_-(X_1,X_2)
  =
 \{p\in Y\,|\, (X_1)_p=-\lambda\cdot(X_2)_p,\,\lambda\in\R\}.
\]
Under some transversality 
assumptions
this set is a codimension-$2$ submanifold of $Y$ and its 
homology class Poincar\'e dual to the obstruction 
class $d^2(X_1,X_2)$ (see~\cite[Lemme~3.2]{Dufraine}). More 
precisely, fixing a trivialization
$\tau$ of $TM$ and a Riemannian metric $g$, the vector fields
$X_1$ and $X_2$ correspond to maps $f_{X_1},f_{X_2}\co M\lra\stwo$.
Define $\Delta=\{(v,-v)\,|\,v\in\stwo\}$, then we demand the
map $(f_{X_1},f_{X_2})\co M\lra\stwo\times\stwo$ to intersect $\Delta$
transversely. Furthermore, we define the 
{\bf homology class of} $X_1$, in symbols $[X_1]$, as the 
homology class of $(f_{X_1})^{-1}(p)$, where $p\in\stwo$ is a regular 
value of $f_{X_1}$. Now suppose that $X_1$ is a nMS vector field 
in a homology class $\frake\in H_1(Y;\Z)$. Furthermore, 
suppose that we obtain $X_2$ from $X_1$ by reversing the 
orientation of a periodic orbit $\gamma$ of $X_1$. The new 
vector field $X_2$ is still Morse-Smale and, in fact, 
\[
 d^2(X_2,X_1)=\pd[C_-(X_2,X_1)]=\pd[\gamma]
\]
(cf.~\cite[Lemme~3.2]{Dufraine} and cf.~Lemma~\ref{lem:transverse}). Recall
that the obstruction $d^2(X_2,X_1)$ can also be written as
 $d^2(X_2,X_1)=\pd[X_2]-\pd[X_1]$ (cf.~\cite[\S 4.2]{Geiges}).

Recall that every Seifert
manifold can be obtained in the following way:
By performing a $(-1/e)$-surgery along a fiber of $\Sigma\times\sone$
we obtain $\fgn$, the $\sone$-bundle over the genus-$g$ base
$\Sigma$ and Euler number $e$. We denote by $\gamma_0$ the
core of the surgery torus. For $r_i\in\Q$, $i=1,\dots,k$, 
such that $r_i=p_i/q_i$ with $p_i\not\in\{-1,1\}$, denote by 
$Y=\fgn(r_1,\dots,r_k)$ the Seifert manifold obtained by performing 
surgeries along $k$ different regular fibers of $\fgn$ with 
coefficients $r_i$. The cores of the surgery tori are called {\bf exceptional
orbits}. We denote them by $\gamma_1,\dots,\gamma_k$. There is
a natural projection map $\pi\co Y\lra\Sigma$. For $i=1,\dots,k$, set 
$p_i=\pi(\gamma_i)$.
By the presentation of oriented Seifert manifolds in terms of surgeries 
 we gave, it is easy to see by a Mayer-Vietoris computation that 
every homology class $\frakc$ can be written as
\[
\frakc=
 \sum_{i=1}^g\lambda_i[\beta_i]
 +\sum_{j=0}^k\alpha_j[\gamma_j],
\]
where the $\beta_i$ are suitable primitive elements on the 
base $\Sigma$ (cf.~also \cite[Lemme~5.3]{Dufraine}). 

Now  suppose we are given 
a Morse-Smale vector field $\xs$ on $\Sigma$ with the
following properties: we have $(\xs)_{p_i}=0$ for $i=0,\dots,k$, and 
for every $\lambda_i\not=0$ the corresponding curve $\beta_i$ is
a periodic orbit of $\xs$. Because Morse-Smale
vector fields exist in abundance on surfaces, it is obvious that
we can find such a vector field. 
Since 
$Y\backslash(\cup_i\nu\gamma_i)$ is a trivial circle bundle we 
also obtain a vector field $X$, there. This vector field can 
be extended over the tubular neighborhoods $\nu\gamma_i$ under
the assumption that the singularities $p_i$ of $\xs$ are either attractive or 
repulsive. The extension will have the property that $\left. X\right|_{\gamma_i}=0$.
The sum $X_0=X+\xf$ is not Morse-Smale, because for 
every periodic orbit $\beta_i$ the vector field
$X_0$ will leave invariant the torus $(\pi)^{-1}(\beta_i)$.
In \cite{Dufraine}, 
a method is sketched to destroy the invariant tori, i.e.~to alter $X_0$ 
in a neighborhood $\nu(\pi)^{-1}(\beta_i)$ of the invariant torus so that 
the new vector field will be Morse-Smale. The destruction of an invariant 
torus creates $2$ additional periodic orbits which lie both in the homology
class $[\beta_i]$.
Thus, after this procedure, the periodic orbits of $X_0$ contain a 
link $L$ such that
\[
 [L]
 =\sum_{\lambda_i\not=0}[\beta_i]
 +\sum_{\alpha_j\not=0}[\gamma_j].
\]
For each $\beta_i$, $\gamma_j$ that appears in this 
equation, we can alter $X_0$ with the 5th operation of Wada
from \cite{Wada}. This operation applied to $\gamma_1$ say
consists of adding two parallel $(p,q)$-cables of $\gamma_1$
to the set of periodic orbits. By choosing $q=\alpha_1$ this
means we obtain a periodic orbit $\widetilde{\gamma}_1$ 
whose homology class equals $\alpha_1[\gamma_1]$. By replacing 
$\gamma_1$ in $L$ by $\widetilde{\gamma}_1$, the homology class of
the new link of periodic orbits fulfills
\[
 [L]
 =\sum_{\lambda_i\not=0}[\beta_i]
 +\alpha_1[\gamma_1]+\sum_{j\not=1,\alpha_j\not=0}[\gamma_j].
\]
Iterating this process, we can change $X_0$ so that it admits
a link of periodic orbits $L$ with $[L]=\frakc$. We construct a
new vector field $X_2$ which is obtained from $X_1$ by reversing
the orientation of the periodic orbits in $L$ (cf.~Lemma~\ref{lem:initial}). 
Then the equality
\[
 d^2(X_2,X_1)=\pd[C_-(X_2,X_1)]=\pd[L]=\pd\thinspace\frakc
\]
holds. 
Moreover, since $C_-(X_1,\xf)$ is empty, we know that 
$d^2(X_1,\xf)=0$. Hence, we have 
\[
 d^2(X_2,\xf)=d^2(X_2,X_1)+d^2(X_1,\xf)=\pd\thinspace\frakc.
\]
We see that for every homology class $\frakc\in H_1(Y;\Z)$ we can construct a
nMS vector field $X$ on $Y$ such that $d^2(X,\xf)=\pd\thinspace\frakc$. 
Furthermore, we can adjust the homotopy class of $X$ without changing 
$d^2(X,\xf)$ (cf.~\cite[Proposition~5.9]{Dufraine}). It is not hard to observe that
this procedure creates $6$ additional periodic 
orbits.

\section{Proof of Theorem~\ref{result}}
We start with the following
observation which can be found in 
\cite[Lemma~3.1]{Yano} and also in \cite[Lemme~5.8]{Dufraine}).
\begin{lem}\label{lem:initial} Given a nMS vector field $X$ with periodic orbit $\gamma$ 
which is either attractive or repulsive then it is possible to alter 
$X$ to a new nMS vector field $X'$ such that it coincides with $X$ 
outside of $\nu\gamma$ and $X'$ has $-\gamma$ as periodic orbit. 
The obstruction class $d^2(X',X)$ equals $\pd[C_-(X',X)]=\pd[\gamma]$.
\end{lem}
The fact that  
$d^2(X',X)=\pd[\gamma]$ was given by Yano in 
\cite[Lemma~3.1]{Yano}. In \cite{Dufraine}, this 
statement is connected with $C_-(X',X)$.
Note that a consideration of $C_-(X',X)$ just makes sense
if the pair $(X',X)$ meets the transversality conditions mentioned
in \S\ref{sec:asodc}. So, 
to relate $[C_-(X',X)]$ with $d^2(X',X)$ in this situation, we have 
to prove that these transversality conditions are fulfilled. This
is, in fact, true. We leave this to the interested reader.
\begin{lem}\label{lem:transverse} Suppose we are given two manifolds 
$Y_i$, $i=1,2$, with boundary and Morse-Smale vector fields $X_i$ 
on $Y_i$. Denote by $K_i$, $i=1,2$, a boundary component of $Y_i$ 
and $\phi\co K_1\lra K_2$ a diffeomorphism. Denote by $Y$ the 
manifold obtained by gluing together $Y_i$, $i=1,2$, with $\phi$. 
Then there is a Morse-Smale vector field $X$ on $Y$ such that 
$\bigl.X\bigr|_{Y_i}$ coincides with $X_i$ outside of a small 
neighborhood of $K_i$.
\end{lem}
\begin{proof} The vector fields $X_1$ and $X_2$ glue together to a smooth 
vector field $X$
on $Y$, which is not necessarily Morse-Smale. There might be stable and
unstable manifolds which do not intersect transversely. Let 
$\gamma_i$, $i=1,2$, be periodic orbits of $X_i$ and suppose 
that $W^u(\gamma_1)\cap W^s(\gamma_2)$ is non-empty. Since $X_i$ 
is transverse to $K_i$, the intersections 
$W^u(\gamma_1)\cap K_1$ and $W^s(\gamma_2)\cap K_2$ are 
both transverse and, thus, $W^u(\gamma_1)\cap K_1$ is a collection
 $L^1_1,\dots,L^1_{k_1}$ of embedded circles in $K_1$. The same 
 is true for $W^s(\gamma_2)\cap K_2$. Let us denote by 
 $L^2_1,\dots,L^2_{k_1}$ the corresponding intersection. The 
 surfaces $K_i$ correspond to a surface in $Y$ we denote by $K$. 
 Then the following is equivalent: 
 $W^u(\gamma_1)\cap W^s(\gamma_2)$ is transverse if and 
only if the intersections $L^1_i\cap L^2_j$ are transverse in $K$ for all
possible choices of $i$, $j$. This is immediate by the observation that
the transversality condition is moved along integral curves by the flow. 
Since the vector field $X_1$ is transverse to the 
surface $K_1$, there exists a collar neighborhood $(-1/2,1/2]\times K_1$
in which $X_1$ corresponds to the vector field $\partial_t$. Similarly, 
since $X_2$ is transverse to the surface $K_2$, there exists a collar neighborhood
$[1/2,3/2)\times K_2$ in which $X_2$ corresponds to the vector field $\partial_t$.
We glue together the
pieces $Y_1$ and $Y_2$ using these collars. Hence, without loss of
generality we may assume to have a neighborhood of $K$ in $Y$ which
is diffeomorphic to
$(-1/2,3/2)\times K$ such that $X$ corresponds to the vector field $\partial_t$, 
i.e.~the canonical vector field in the first coordinate. Now suppose that 
$L^1_i$ and $L^2_j$  do not intersect transversely for some $i$ and $j$. Then 
there is
an isotopy $\varphi_t$ of the surface $K$ which will make the intersections
transverse by a deformation of $L^1_i$. We can assume that $\varphi_t$ is 
the identity for small $t<\epsilon$ and
that $\varphi_t$ is independent of $t$ for $t>1-\epsilon$.
On the piece $[0,1]\times K$ we define the vector field by
\[
  X'_{(t,\varphi_t(p))}
 =
 \partial_t
 +
  \frac{d\varphi_{t}(p)}{dt}.
\]
By
its definition, the flow $\Phi$ of this vector field fulfills $\Phi_{t}(p)=(t,\varphi_t(p))$
for all $p\in K$ and times $t\in\R$. At $(\epsilon/2,1-\epsilon/2)\times K$
we replace $X$ by $X'$. Now, $W^u(\gamma_1)$ intersects $W^s(\gamma_2)$
transversely: To see this, we just have to check that the
intersection of the set $W^u(\gamma_1)\cap (\{1\}\times K)$ with
$\{1\}\times L^2_j=W^s(\gamma_2)\cap(\{1\}\times K)$ is transverse. But, by construction,
we have
\[
  W^u(\gamma_1)\cap (\{1\}\times K)
  =
 \Phi_1(\{0\}\times L^1_i)
 =
 \{1\}\times\varphi_1(L^1_i)
\]
which intersects $\{1\}\times L^2_j$ transversely.
\end{proof}
We now discuss a method to destroy the invariant 
tori over the $\beta_i$ such that we can spare the 5th Wada 
operation on them (cf.~\S\ref{sec:asodc}).
\begin{prop}\label{prop:destroy_torus} It is possible to destroy 
the invariant torus over $\beta_i$ by introducing two new periodic 
orbits which both represent the homology class $\lambda_i[\beta_i]$ 
such that the new vector field is sill nMS. 
\end{prop}
\begin{proof} The proof consists of two steps. In the first step 
we give the construction and in the second step we prove that the 
new vector field is still nMS.

Let $\gamma$ be a closed orbit of $X_\Sigma$ on the base space. 
There is a neighborhood $U$ of $\gamma$ in $\Sigma$ such that 
$U\cong\sone\times[-1,1]$ with coordinates $(t,x)$. In these coordinates 
$X_\Sigma$ corresponds to the vector field $\partial_t-x\partial_x$. Hence, 
$V=\pi^{-1}(U)\subset Y$ is diffeomorphic to 
$\sone\times[-1,1]\times\sone$ with coordinates $(t,x,z)$ such 
that $X_0$ corresponds to 
$\partial_t-x\thinspace\partial_x+\partial_z$. In these coordinates 
the invariant torus is $T=\sone\times\{0\}\times\sone$. Now consider 
the vector fields
\[
 \begin{array}{rcl}
  \xpar&=&\lambda_i\partial_t+\partial_z\\
  \xort&=&-\partial_t+\lambda_i\partial_z
\end{array}
\]
and a homeomorphism $\phi\co\sone\times\sone\lra T$ which 
sends a meridian $\mu=\sone\times\{*\}$ to a $(\lambda_i,1)$-curve 
and a longitude $\{*\}\times\sone$ to a $(-1,0)$-curve in $T$. Define 
a smooth function $h\co[0,1]\times[0,1]\lra\R$ by 
$h(a,b)=\cos(2\pi\thinspace b)$ and a function $g\co T\lra\R$ 
by $g:=h\circ\phi^{-1}$. Extend $g$ to a function
\[
 g\co
 \sone\times[-1,1]\times\sone\lra\R
\]
such that $g(t,\pm1,z)=(\lambda_i-1)/(\lambda_i^2+1)$. Furthermore, 
define a function
\[
 f
 \co
 \sone\times[-1,1]\times\sone\lra\R
\]
such that $\left.f\right|_T\equiv1$ and 
$\left.f\right|_{\sone\times\{0\}\times\sone}=1-\lambda g$. With this 
at hand, we consider the vector field
\[
 \xline_{(t,x,z)}=f\thinspace\xpar-x\partial_x+g\thinspace\xort.
\]
This vector field has the following properties: On the torus $T$ the 
flow of the vector field $\xline$ admits two periodic orbits which 
are both $(\lambda_i,1)$-curves. The vector field 
$\xline$ can be extended smoothly 
to the manifold $Y$ by setting $\xline_p=(X_0)_p$ for $p\in Y\backslash V$, because 
for $(t,x,z)\in\sone\times\{\pm1\}\times\sone$ we have
\[
 \xline_{(t,x,z)}=\partial_t-x\partial_x+\partial_z.
\]
In order to see that $\xline$ is nMS, we have to check that the stable 
and unstable manifolds of the periodic orbits of $\xline$ intersect 
transversely. With a small perturbation of $\xline$ in the neighborhood of
the boundary $\sone\times\{\pm1\}\times\sone$ this can be achieved
(see~Lemma~\ref{lem:transverse}).
\end{proof}
\begin{figure}[t!]
\definecolor{myred}{HTML}{B00000}
\definecolor{myblue}{HTML}{0000B0}
\definecolor{mygreen}{HTML}{00B000}
\definecolor{mymagenta}{HTML}{b100b0}
\labellist\small\hair 2pt
\pinlabel{Remove tubular neighborhoods of the $\mu_i$ and cap off with 
disks} [b] at 800 -10
%
%
\pinlabel{{\color{mygreen} $0$}} [t] at 948 85
\pinlabel{{\color{mygreen} $0$}} [t] at 1146 85
\pinlabel{{\color{mymagenta} $0$}} [t] at 1350 85
\pinlabel{{\color{mymagenta} $0$}} [t] at 1550 85
%
\pinlabel{{\color{myred} $2$}} [b] at 948 597
\pinlabel{{\color{myred} $2$}} [b] at 1146 597
\pinlabel{{\color{myblue} $2$}} [b] at 1345 597
\pinlabel{{\color{myblue} $2$}} [b] at 1544 597
%
%
\pinlabel{$1$} [b] at 1050 490
\pinlabel{$1$} [b] at 1250 490
\pinlabel{$1$} [b] at 1450 490
\pinlabel{$1$} [t] at 1050 190
\pinlabel{$1$} [t] at 1250 190
\pinlabel{$1$} [t] at 1450 190
\pinlabel{{\color{myblue} $p_1$}} [b] at 455 597
\pinlabel{{\color{myblue} $p_2$}} [b] at 654 597
\pinlabel{{\color{mymagenta} $p_3$}} [t] at 455 85
\pinlabel{{\color{mymagenta} $p_4$}} [t] at 654 85
\pinlabel{{\color{myred} $\mu_1$}} [r] at 130 620
\pinlabel{{\color{mygreen} $\mu_2$}} [r] at 130 62
\endlabellist
\centering
\includegraphics[width=15cm]{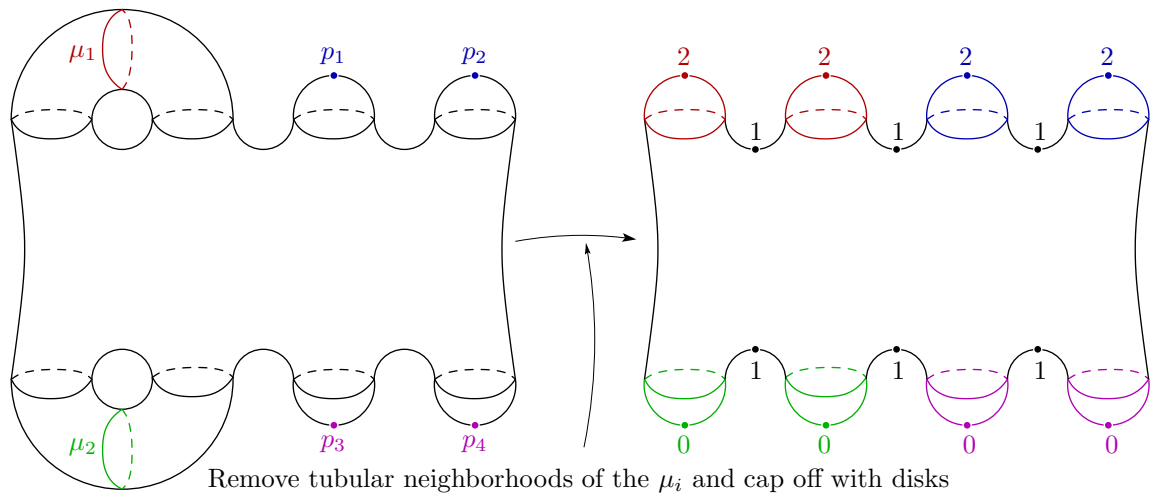}
\caption{We remove tubular neighborhoods of the $\mu_i$, $i=1,\dots,g$, and
then cap off the boundary components with disks. There, we define a Morse 
function $f$ whose gradient admits the singularities as indicated in the 
right of this Figure with the numbers indicating the indices of the singularities.}
\label{Fig:proof}
\end{figure}
\begin{proof}[Proof of Theorem~\ref{result}] 
Given a Seifert manifold
$Y=\fgn(r_1,\dots,r_k)$, pick a homology class $\frakc$ which is
given by
\[
 \frakc
 =
 \sum_{i=1}^g\lambda_i[\beta_i]
 +
 \sum_{j=0}^k\alpha_j[\gamma_j]
 =
 \sum_{i=1}^g\lambda_i[\beta_i]
 +
 \alpha_0[\gamma_0]
 +
 \sum_{j=1}^k\alpha_j[\gamma_j]
 , 
\]
where $\lambda_i$, $i=1,\dots,g$, and $\alpha_j$, $j=0,\dots,k$ are
all elements in $\Z\backslash\{-1,0,1\}$. Let us call these 
classes {\bf maximal}. We separated the summand $\alpha_0[\gamma_0]$
from the other $[\gamma_j]$'s,  because this element just appears for
$e\not=\pm 1$. We now fix a maximal class $\frakc$. For every homology 
class $\frakh\in H_1(Y;\Z)$
we may apply the construction process described in \S\ref{sec:asodc} to 
get a nMS vector field in that homology
class. Its number of periodic orbits shall be denoted by $n(Y,\frakh)$.
Then it is easy to see that
$
 n(Y,\frakc)
 \geq
 n(Y,\frakh)
$
for all $\frakh\in H_1(Y;\Z)$. Furthermore, for every other maximal homology class
$\mathfrak{d}$ we see that $n(Y,\mathfrak{d})=n(Y,\mathfrak{c})$. 
Recall from \S\ref{sec:asodc} that adjusting the homotopy class of a vector field
(while fixing its homology class) creates $6$ additional periodic orbits.
Hence, $n(Y)\leq n(Y,\frakc)+6$ which provides
an upper bound for $n(Y)$. The crucial ingredient is to
generate a Morse-Smale flow on the base $\Sigma$ with periodic orbits in
the homology classes of $\pi(\beta_i)$, $i=1,\dots,g$ and singularities at 
the points $\pi(\gamma_j)$, $j=0,\dots,k$: Recall that a 
nMS vector field can be given by a round handle decomposition. 
Furthermore, observe that
a Morse-Smale vector field without periodic orbits can be viewed as the
gradient of a Morse function, which in turn defines a handle decomposition.
Hence, to provide a Morse-Smale vector field on $\Sigma$ with the
desired non-wandering set, we will cut the surface into two pieces, 
$\Sigma=F_0\cup_\partial F_1$, where $F_0$ is given by a round handle
decomposition and $F_1$ by a handle 
decomposition (or, equivalently, by a Morse function). Hence, we will get 
Morse-Smale vector fields $X_i$, $i=0,1$, such that $X_0$ is non-singular and
$X_1$ has singularities but no periodic orbits. In this way, we have
precise control of the non-wandering set of $X_0\cup_\partial X_1$. To 
provide the upper bound, we have to move through a couple of different cases.
We will discuss the first case thoroughly to point out that the proof
technique works and then discuss the other cases in a brief way:\vspace{0.3cm}\\
\noindent{\sc Case 1~--~$g>0$, $e\not=\pm 1$, $n\geq0$:} The surface $\Sigma$ is
of genus $g$ and, hence, can be written as a connected sum of $g$ tori
$T^2_1\#T^2_2\#\dots\#T^2_g$. Furthermore, $\pi(\beta_i)$, $i=1,\dots,g$, 
sits in the torus component $T^2_i$. Hence, after applying a suitable 
diffeomorphism $\phi\co\Sigma\lra\Sigma$, we may think of $\pi(\beta_i)$ as
a meridian $\mu_i$ of $T^2_i$. Removing tubular neighborhoods $\nu\mu_i$, 
$i=1,\dots,g$, we obtain a sphere with $2g$ holes, $S$ say. The holes
of the surface should be grouped in pairs such that the pairs belong
to the same meridian $\mu_i$. We cap off the boundary components
with disks and obtain a sphere as indicated in Figure~\ref{Fig:proof}. 
We define a Morse function $f$ on this sphere with critical points
of index $0$ or $2$ at the singularities $p_i$, $i=0,\dots,k$, and with
singularities of index $0$ or $2$ at the centers of the capping disks. 
The function should be
defined in such a way that the gradient is Morse-Smale. In other words, 
there should not be any separatrices connecting critical points of
index $1$. To demonstrate 
that this can be realized, we provide a handle decomposition in
Figure~\ref{Fig:proof2}.  This figure shows a handle decomposition of 
the sphere with $1$-handles which are not attached to each other.
\begin{figure}[t!]
\labellist\small\hair 2pt
\endlabellist
\centering
\includegraphics[height=8cm]{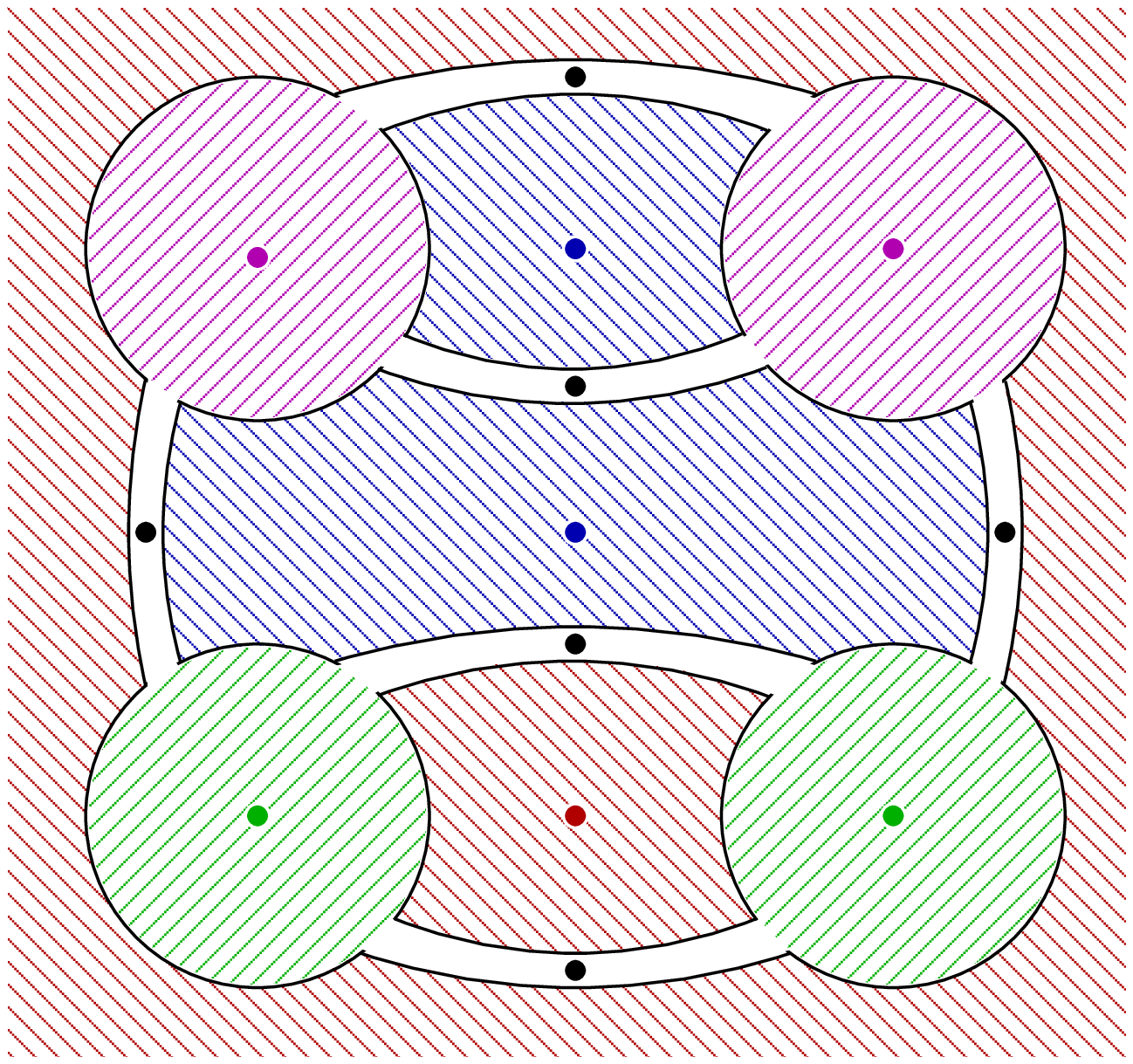}
\caption{A handle decomposition of $\stwo$. The coloring coincides with
the coloring given in Figure~\ref{Fig:proof}. The white strips are the
$1$-handles in this handle decomposition. The associated Morse function
has a Morse-Smale gradient. The points are the singularities of the gradient
where one red point sits at infinity. The indices of the singularities are
given in Figure~\ref{Fig:proof}.
}
\label{Fig:proof2}
\end{figure}
Then we set $\Bigl.X_1\Bigr|_S= \Bigl.(\phi^{-1})_*(\nabla f)\Bigr|_S$.
Observe that $\Sigma\backslash S$ is a union of cylinders, i.e.~of
$2$-dimensional round handles. On these round handles we define $X_2$ to
equal the standard model of a Morse-Smale flow on a round handle: More 
precisely, every component of $\Sigma\backslash S$ is diffeomorphic to
$\sone\times[-1,1]$ with coordinates $t$ and $x$. On each of these 
components we require the vector field $X_2$ to equal $\partial_t-x\partial_x$
if $\mu_i$ is an attracting periodic orbit and $\partial_t+x\partial_x$ if
it is a repelling periodic orbit. The vector field $X_\Sigma$, which is 
obtained by gluing together $X_1$ and $X_2$, is a Morse-Smale vector field with
periodic orbits the $\mu_i$, $i=1,\dots,g$, with singularities
the $\pi(\gamma_j)$, $j=0,\dots,n$, which are either attracting or
repelling, and with $2g+n-1$ singularities which are all saddles. To 
this vector field $X_\Sigma$ we 
apply the algorithm described in the previous section. We first lift
this vector field to a vector field on the Seifert manifold. This lift
admits $g$ invariant tori, $n+1$ singularities which are either attracting
or repelling, and $2g+n-1$ saddle orbits. The destruction of the invariant
tori (in the sense of Proposition~\ref{prop:destroy_torus}) and the 
application of the 5th operation of Wada (cf.~\S\ref{sec:asodc}) leaves us with
$ n(Y,\frakc)=2g+3(n+1)+2g+n-1$
periodic orbits. Finally, we adjust the homotopy class (within
its homology class) which creates $6$ additional periodic orbits. We get
\[
 n(Y)\leq n(Y,\frakc)+6=2g+3(n+1)+2g+n-1+6=4g+4n+8,
\]
which ends the first case.\vspace{0.3cm}\\
\noindent {\sc Case 2~--~$g>0$, $e=\pm1$, $n\geq0$:} This case differs from
the first just by its Euler number. If the Euler number is $\pm1$, then the
regular fibers are nullhomologous. Hence, for every class $\frakc$ the coefficient
$\alpha_0$ will vanish. We proceed as before and generate $X_\Sigma$ with
$g$ periodic orbits, $n$ singularities which are attractors/repellors and
$2g+n-2$ saddles. Then we lift this vector field, destroy the invariant
tori, perform the 5th operation of Wada and get
\[
 n(Y)\leq n(Y,\frakc)+6=2g+3n+2g+n-2+6=4g+4n+4,
\]
which completes the second case.\vspace{0.3cm}\\
{\sc Case 3~--~$g=n=0$, $e\not=\pm1$:} The base space is a sphere and
we need an arbitrary Morse-Smale vector field on the base, i.e.~we put no
requirements on the set of periodic orbits or the set of singularities. However, 
every gradient of a Morse
function on the sphere has at least two singularities. So, we pick
such a gradient and perform the algorithm presented in the previous section.
The vector field lifts to a Morse-Smale field on the Seifert manifold with
$2$ periodic orbits. We have to apply the 5th operation of Wada on one
of these orbits and then we have to adjust the homotopy class (within its
homology class) which generates additional $6$ periodic orbits.
We obtain
$
 n(Y)\leq 3+1+6=10
$.\vspace{0.3cm}\\
{\sc Case 4~--~$g=n=0$, $e=\pm1$:} We pick the same vector field on the
base as in the third case. The lift is Morse-Smale with two periodic
orbits. We adjust the homotopy class which generates $6$ additional periodic
orbits. We obtain $n(Y)\leq 2+6=8$.

\end{proof}

\section{Extension to Graph Manifolds}
Recall that a $3$-dimensional manifold $Y$ is called {\bf graph manifold} if
its prime decomposition consists of manifolds $Y_1,\dots,Y_n$ such that in
every $Y_j$, $j=1,\dots,n$, there exists a minimal collection of disjointly 
embedded tori such that the complement of these 
tori is a disjoint union of Seifert manifolds $\ipiece$, $i=1,\dots,l$. Here, 
$\ipiece$ denotes a Seifert manifold with genus-$g_i$ base with $k_i$
boundary components and $n_i$ exceptional orbits.
Such 
a decomposition into Seifert manifolds is called a {\bf JSJ decomposition}. 
We will first restrict to irreducible graph manifolds $Y$. 
To derive an upper bound we follow the natural approach: We will produce 
nMS vector fields on $Y$ by gluing together nMS vector fields from the Seifert 
pieces $\ipiece$ together. 
There are two constructions we have to provide: We have to generate
a reference vector field to determine the homology classes of vector 
fields (cf.~\S\ref{sec:asodc}). And furthermore, we have to generate
nMS vector fields on the pieces $\ipiece$ in such a way that these
vector fields glue together to a vector field on $Y$ that is nMS.
\subsection{The Reference Vector Field}\label{sec:trvf} Recall that for 
closed Seifert manifolds there
exists a surgical presentation as introduced in \S\ref{sec:asodc}. For a Seifert manifold
with boundary we can also provide such a description in an analogous manner.
Just note, that the base is
a surface with boundary and that all $\sone$-bundles over such
are trivial.
For every Seifert piece $\ipiece$ we pick such a presentation. 
The manifold $\ipiece$ admits a natural vector field $\xf$ which is 
tangent to the fibers. The presentation of $\ipiece$ we have chosen 
induces a preferred collar neighborhood for all of its boundary components.
Given such a boundary component, we have a preferred identification of a
neighborhood with $[0,1]\times\sone$. Denote by $\partial_r$ the vector field
in the direction of the interval. So, in a collar neighborhood of the boundary 
we perturb $\xf$ to
$X_0=f\cdot \xf + g\cdot \partial_r$, 
where $f$ is a non-negative function which is zero in a neighborhood of
the boundary and increases to one as we move away from the boundary.
And $g$ is a non-negative function which behaves in the opposite way
as $f$, i.e.~it is zero away from the boundary and it smoothly increases to one 
as we approach the boundary. We perform this construction with every Seifert
piece and then glue the vector fields together to obtain one on the 
manifold $Y$. We will denote this vector field by $X_0$. This will serve us as
our reference vector field.
\subsection{The Construction Method} Constructing nMS vector fields on Seifert 
pieces with boundary
in principle works the same way as done in the previous section. By a 
Mayer-Vietoris argument it is not hard to see that every homology class
$\alpha$ can be written as
\begin{equation}
 \alpha
 =
\sum_{a=1}^{g_i}\lambda_a\cdot[\beta_a]
+
\sum_{b=0}^{n_i}\alpha_b[\gamma_b]
+
\sum_{c=1}^{k_i-1}\tau_c[\delta_c],
\label{eq:rep}
\end{equation}
where the $\beta_a$ are primitive elements in the homology of the surface,
the $\gamma_b$, $b\not=0$, are the exceptional orbits, $\gamma_0$
is a regular fiber, and the $\delta_c$ are closed curves parallel
to the boundary components of the surface. We proceed as before and
create a Morse-Smale vector field $X_\Sigma$ on the base 
with the points $\pi(\gamma_1),\dots,\pi(\gamma_{n_i})$ as
singularities and the curves $\beta_1,\dots,\beta_{g_i}$ and 
$\delta_1,\dots,\delta_{k_i-1}$ as periodic orbits.
We lift this vector
field with the procedure given in the previous sections 
(using $\left. X_0\right|_{\ipiece}$ instead of
$\xf$) and then destroy the 
invariant tori (in the sense of Proposition~\ref{prop:destroy_torus}) over 
both the $\beta$-curves and the $\delta$-curves. The
destruction is done such that we create periodic orbits whose homology
classes represent $\lambda_a\cdot[\beta_a]$ and $\tau_c\cdot[\delta_c]$.
Then the 5th operation of Wada will allow us to replace 
$\gamma_b$, $b=0,\dots,n_i$,
by a periodic orbit which represents the class 
$\alpha_b\cdot[\gamma_b]$, $b=0,\dots,n_i$. Finally, we make the vector field
transverse to the boundary components like done for the reference vector field.
Performing this construction for every piece $\ipiece$, these vector fields
can be glued together to provide a nMS vector field on $Y$ whose set
of periodic orbits contains a link $L$ consisting of attractive or repulsive
orbits such that $[L]=\frakc$ for every given $\frakc$.
\subsection{The Proofs of the Main Results}
Before delving into the proof we would like to state the following
result of Yano which will be used in the proof of 
Theorem~\ref{thm:graphmanifolds}.
\begin{prop}[Theorem~1 of \cite{Yano}]\label{recover} Let $Y$ be a graph manifold prime to $\stwo\times\sone$
and $\rho\co Y\lra\mathcal{C}_Y$ the natural map into the 
Jaco-Shalen-Johansson complex of $Y$. Then the homotopy class 
of a vector field $X$ admits a (non-singular) Morse-Smale representative if and 
only if  $\rho_*(e(X))$ vanishes in $H_1(\mathcal{C}_Y;\Z)$, where $e(X)$ is
the  Poincar\'{e} dual of the Euler class of the field of $2$-planes orthonormal to $X$.
\end{prop}
Furthermore, note that there are also homotopical
invariants for vector fields on manifolds with boundary 
(cf.~for instance \cite[p.~439]{Yano}). In \S\ref{sec:asodc} we briefly
introduced the geometric interpretation of homology classes of vector
fields. This is based on the Pontryagin construction which is described
in \cite[\S 7]{Milnor}.
The Pontryagin construction as described in \cite[\S 7]{Milnor} can be 
adapted to work for the case of manifolds with boundary. Then for
two vector fields $X_1$, $X_2$ on a manifold $Y$ with non-trivial
boundary, the homology class of $C_-(X_1,X_2)$ measures the
homotopical distance away from the $3$-cells of $Y$ as in the closed
case. However, in the case of non-empty boundary, the vector fields
are understood to be in a pre-chosen homotopy class along the boundary
of $Y$ (cf.~\cite[p.~439]{Yano} and \cite[Lemma~1.5]{Yano}). 

Because the homotopical classification of vector fields on manifolds with
boundary works the same way as in the closed case, we can apply
our algorithm to the Seifert pieces $Y_j$ of a JSJ decomposition of a graph
manifold to obtain upper bounds for the $n(Y_j)$.
\begin{proof}[Proof of Theorem~\ref{thm:graphmanifolds}] 
The homology
classes in the kernel of $\rho_*$ are contained in the image of the map
\begin{equation}
 \bigoplus_{i=1}^l H_1(\ipiece;\Z)\lra H_1(Y;\Z)
\label{vietoris}
\end{equation}
given by the obvious Mayer-Vietoris sequence (cf.~\cite[p.~444]{Yano}). Denote by $\frake$ the homology
class of the reference vector field $X_0$ (cf.~\S\ref{sec:trvf}). By the considerations
from above we can glue together nMS vector fields on the pieces to obtain 
a nMS vector field on $Y$. If we do this as before, we can define a nMS vector field
$X$ on $Y$ such that
$C_-(X,X_0)$ is empty and, hence, 
$\frake$ can be written as a 
sum $\frake_i\in H_1(\ipiece;\Z)$, $i=1,\dots,l$ (cf.~Proposition~\ref{recover}).
So, suppose we are given
a class $\frakc\in H_1(Y;\Z)$ which can be written as a sum of classes
$\frakc_i\in H_1(\ipiece;\Z)$, $i=1,\dots,l$. Then by the previous discussion,
we see that on every piece $\ipiece$ there exists a nMS vector field $X_i$ whose
set of periodic orbits contains a link $L_i$ consisting of attracting and
repelling periodic orbits such that $[L_i]=\frakc_i-\frake_i$. We glue the $X_i$ together
to obtain a nMS vector field $X$ on $Y$. Hence, the set of periodic orbits of $X$
contains the link $L=L_1\cup\dots\cup L_l$. Then $C_-(X,X_0)$ is empty and, so, 
they lie in the same homology class. We generate a new nMS vector field by 
reversing the orientation of the periodic orbits contained in $L$. Denote the
new vector field by $X'$. Then, we have 
\[
 [X']=[C_-(X',X_0)]+[X_0]=[L]+\frake=\frakc-\frake+\frake=\frakc.
\]
The maximal
number of periodic orbits will be given when choosing a class 
$\frakc-\frake$ whose presentation in the form of Equation~\eqref{eq:rep} 
has the property that 
$\lambda_a\not\in\{-1,0,1\}$ for all $a=1,\dots,g_i$, 
$\gamma_b\not\in\{-1,0,1\}$ for 
$b=0,\dots,n_i$, and $\delta_c\not\in\{-1,0,1\}$ for $c=1,\dots,k_i-1$. So, a maximal
class in the graph manifold is a sum of maximal classes of the Seifert pieces.
Hence, our previous considerations, i.e.~especially the proof 
Theorem~\ref{result}, provides us with the upper bound
\[
 n(\ipiece)
 \leq
 4g+4n+8+2\delta_{g_i,0}\delta_{n_i,0}
+2(k_i-1)
\]
for every $i=1,\dots,l$.
Observe that in this bound we already included the changes to adapt the homotopy
classes within a fixed homology class. Hence, for $l>1$ we have
\[
n(Y)\leq
6+\sum_{i=1}^l
n(\ipiece)-6,
\]
which is equivalent to
$
 n(Y)
 \leq
6+2\cdot
\sum_{i=1}^l
\bigl(
2g_i+2n_i+\delta_{g_i,0}\delta_{n_i,0}+k_i
\bigr).
$
\end{proof}
\begin{proof}[Proof of Corollary~\ref{cor:graphmanifolds}] This statement immediately
follows from our discussion and the observation that it is possible to define 
connected sums of nMS vector fields on manifolds and that their homotopical 
invariants behave additive under connected sums. This was observed by 
Yano in \cite[\S 2]{Yano} (especially \cite[Proposition~2.8]{Yano}).
\end{proof}

\end{document}